\documentclass[12pt]{amsart}
\usepackage[all]{xy}
\usepackage[parfill]{parskip}

\usepackage{verbatim}
\usepackage{color}

\usepackage{amsmath, amscd, graphicx, latexsym, hyperref, times}
\usepackage{esint} 
\usepackage{graphicx}
\usepackage[abs]{overpic}
\usepackage{tikz}
\usepackage{tikz-cd}
\usetikzlibrary{arrows, patterns}

\oddsidemargin .25in \theoremstyle{plain}
\newtheorem{theorem}{Theorem}

\newtheorem{lemma}{Lemma}
\newtheorem{proposition}{Proposition}

\numberwithin{equation}{section}
\numberwithin{lemma}{section}
\numberwithin{proposition}{section}
\numberwithin{corollary}{section}
\numberwithin{remark}{section}
\allowdisplaybreaks[3]

\usepackage{cleveref}
\crefname{theorem}{Theorem}{Theorem}
\crefname{lemma}{Lemma}{Lemmas}
\crefname{proposition}{Proposition}{Propositions}
\crefname{question}{Question}{Questions}
\crefname{conjecture}{Conjecture}{Conjectures}
\crefname{equation}{}{}
\crefname{section}{Section}{Section}
\crefname{figure}{Figure}{Figure}
\crefname{theoremx}{Theorem}{Theorem}
\usepackage{appendix}
\begin{document}

\title[A necessary condition for the logarithmic Minkowski problem]{A necessary condition for the logarithmic Minkowksi problem in higher dimension}

\author{Mijia Lai}
\address{School of Mathematical Sciences, Shanghai Jiao Tong University, Shanghai 200240, China}
\email{laimijia@sjtu.edu.cn}

\author{Zixiao Wang}
\address{Yangpu Senior High School, Shanghai 200092, China}
\email{3919034156@qq.com}

\thanks{This paper grows out from a research project in the National Scientific and Technological Innovation Talent Training Program of Shanghai Municipality.}
\date{}
\begin{abstract}
In this paper, we establish a necessary condition for the logarithmic Minkowski problem in higher dimensions. This result generalizes a necessary condition proposed by Liu, Lu, Sun, and Xiong in their investigation of the two-dimensional case, and also refines the so-called subspace concentration condition.
\end{abstract}
\maketitle
\section{Introduction}
The classical Minkowski problem, dating back to 1897 by Hermann Minkowski \cite{Mi}, asks for the existence of a strictly convex closed surface $S$ in $\mathbb{R}^3$ whose Gaussian curvature, when expressed as a function of the unit normal, coincides with a prescribed positive function $f$ on the sphere $\mathbb{S}^{2}$ . Reflecting on its profound influence, Calabi once remarked that "from the geometric viewpoint, it [the Minkowski problem] is the Rosetta Stone, from which several related problems can be solved."

In its modern, measure-theoretic formulation, the Minkowski problem seeks necessary and sufficient conditions for a given Borel measure $\mu$ on $\mathbb{S}^{n-1}$ to be the surface area measure of a convex body $K \subset \mathbb{R}^n$. For a convex body $K$ with boundary $\partial K$, denote by
\[
\nu: \partial' K \rightarrow \mathbb{S}^{n-1}
\]
the Gauss map, which is defined $a.e. \mathcal{H}^{n-1}$ on $\partial K$ where the normal is unique. The surface area measure $\mu_K$ of $K$ is then defined on Borel subsets $E \subset \mathbb{S}^{n-1}$ by
$$
\mu_K(E):=\mathcal{H}^{n-1}\left(\nu^{-1}(E)\right),
$$
where $\mathcal{H}^{n-1}$ denotes the $(n-1)$-dimensional Hausdorff measure. 

Minkowski himself first solved the polytopal case of the problem \cite{Mi}. Let $n_1, \ldots, n_m$ be distinct, non-coplanar unit vectors in $\mathbb{R}^3$ and $a_1, \ldots, a_m>0$ be positive real numbers. Then there exists a closed convex polyhedral surface in $\mathbb{R}^3$, unique up to translation, whose faces have precisely these vectors as outer unit normals and these numbers as their respective surface areas, if and only if
$$
\sum_{i=1}^m a_i n_i=0 .
$$
The condition $\sum_i a_i n_i=0$ admits a clear geometric interpretation: the total signed area of the projection of any closed polyhedral surface onto an arbitrary plane must be zero.

Subsequently, Bonnesen-Fenchel\cite{BF} and Alexandrov\cite{Al} extended the existence and uniqueness result to general convex hypersurfaces in all dimensions. Later, the regularity of solutions-especially the question of when a weak (generalized) convex surface is in fact smooth-was established through seminal contributions of (among many others) Lewy \cite{Le}, Nirenberg\cite{Ni}, Cheng-Yau\cite{CY} and Pogorelov \cite{Po}.

Lutwak\cite{Lu}, building on earlier ideas of Firey \cite{F}, introduced the notion of $L_p$ surface area measure for a general convex body $K \subset \mathbb{R}^n$. For a Borel set $E \subset \mathbb{S}^{n-1}$, the $L_p$ surface area measure $\mu_{(K, p)}$ is defined by
$$
\mu_{(K, p)}(E):=\int_{\nu_K^{-1}(E)}\left(x \cdot \nu_K(x)\right)^{1-p} d \mathcal{H}^{n-1}(x).
$$
Note that when $p=1$, $\mu_{(K, 1)}$ is just the surface area measure of $K$ mentioned above.

Analogous to the classical Minkowski problem, the $L_p$ Minkowski problem asks for necessary and sufficient conditions on a given Borel measure $\mu$ on $\mathbb{S}^{n-1}$ to be the $L_p$ surface area measure of some convex body.
In the past two decades, significant progress has been made on this problem-see, for example, the foundational works of Lutwak \cite{Lu}, Chou-Wang \cite{CW}, Lutwak-Yang-Zhang \cite{LYZ}, and Böröczky-Lutwak-Yang-Zhang \cite{BLYZ}. The reader is also referred to the excellent survey by Huang-Yang-Zhang\cite{HYZ} on the theory of Minkowski problems for variety of geometric measures including the above mentioned $L_p$ surface area measure. 

An important special case of the $L_{p}$ Minkowski problem arises when $p=0$, known as the logarithmic Minkowski problem. The associated $L_{0}$ surface measure-commonly referred to as the cone-volume measure (up to a factor of $\frac{1}{n}$)-is geometrically natural. For a polytope $P$ containing the origin in its interior, the cone-volume measure of a face is simply the volume of the convex hull of that face and the origin.

In a seminal work\cite{BLYZ}, Böröczky-Lutwak-Yang-Zhang established that the logarithmic \newline 
Minkowski problem admits a solution for an even measure if and only if the measure satisfies the subspace concentration condition. 
Later, Chen-Li-Zhu\cite{CLZ} extended this result by proving that the same subspace concentration condition remains sufficient for general (non-symmetric) measures.

The discrete version of the logarithmic Minkowski problem is to prescribe a disrete  cone volume measure of the form $\mu= \sum_{i=1}^m v_i \delta_{n_i}$. Geometrically, this amounts to finding necessary and sufficient conditions on a set of unit vectors $n_1, \dots, n_m \in \mathbb{R}^n$ and a set of positive numbers $v_1, \cdots, v_m>0 $ that ensure the existence of an $m$-faced polytope $P$, containing the origin in its interior, whose faces have outer unit normals $n_1, \cdots, n_m$ and corresponding cone-volumes $v_1, \cdots, v_m$.

Stancu \cite{S1} investigated the discrete logarithmic Minkowski problem in $ \mathbb{R}^2$, i.e., the case of polygons. She established a general existence result for $n$-gons that have no pair of parallel sides. Additionally, she provided a sufficient condition for existence in two further cases: either when $n>4$ and the polygon has at least two pairs of parallel sides, or when the polygon is a trapezoid.

Zhu~\cite{Z} solved the discrete logarithmic Minkowski problem for measures whose supports are in general position. Here, a set of unit vectors in $\mathbb{R}^n$ is said to be in general position if they are not contained in any closed hemisphere and any $n$ of them are linearly independent. Later, Böröczky-Hegedűs-Zhu~\cite{BHZ} proved that the essential subspace concentration condition is also sufficient for the existence of solutions to the discrete logarithmic Minkowski problem in $\mathbb{R}^n$. 

Liu-Lu-Sun-Xiong ~\cite{LLSX} recently discovered a necessary condition for two dimensional convex bodies whose centroid lies at the origin. They also completely solved the logarithmic Minkowski problem for quadrilaterals.

The main result of the present paper is to extend above-mentioned Liu-Lu-Sun-Xiong's condition to higher dimensions.  

\begin{theorem} \label{main}
    Let $K\subset \mathbb{R}^n,(n\geq 3)$ be a convex body with \textbf{centroid at the origin}, $\mu$ be the associated cone volume measure. Set $V:=\mu(\mathbb{S}^{n-1})=V(K)$, then for any $u\in \mathbb{S}^{n-1}$, there holds
    \begin{align} \label{necessary}
    n\frac{\mu(u)+\mu(-u)}{V}+(n+1)^{n-1} |\frac{\mu(u)-\mu(-u)}{V}|^n\leq 1.
    \end{align}
    Moreover if the equality holds for some $u$, then $K$ is 
    \begin{itemize}
        \item either an oblique prism of the form $K=\operatorname{conv}\{ P\cup P+\vec{l}\}$,
        \item or a cone of the form $K=\operatorname{conv}\{P \cup \{p\}\}$.
    \end{itemize}
    Here $P$ is an $(n-1)$-dimensional convex body lying in an $(n-1)$-dimensional hyperplane $H$ perpendicular to $u$, and $\vec{l}$ is a translation vector and $p$ a point both not contained in $H$. $\operatorname{conv}\{A\}$ represents the convex hull of $A$, i.e., the smallest convex set containing $A$. 
\end{theorem}

\begin{figure}
    \includegraphics[width=0.3\linewidth]{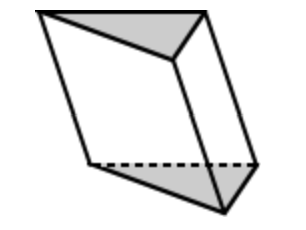} 
      \includegraphics[width=0.25\linewidth]{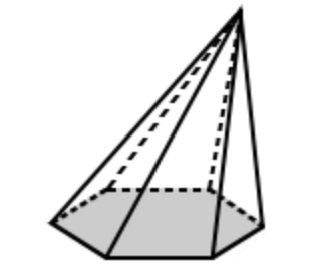}
    \caption{Left: oblique prism, Right: cone}
\end{figure}

We make some remarks on our condition. First, for 
$n=2$, condition \eqref{necessary} coincides with that of Liu–Lu–Sun–Xiong. Moreover, if the equality holds for some $u$, then $K$ is either a trapezoid (including a parallelogram or rectangle as a special case) with its two parallel sides perpendicular to $u$, or a triangle with one side perpendicular to $u$. This situation differs slightly from the equality case for $n\geq 3$. While the cone case is analogous to the triangle, the prism case is more strict than the trapezoid case.  An oblique prism has two congruent bases differing only by a translation. Its two dimension analogue is necessarily a parallelogram.  A trapezoid is the convex hull of its two parallel sides, which can have different length. 

Secondly, \eqref{necessary} actually refines the subspace concentration condition for one-dimensional subspaces, as we now explain.

A finite Borel measure $\mu$ on $S^{n-1}$ is said to satisfy the subspace concentration condition, if for every subspace $\xi$ of $\mathbb{R}^n$ with $0<\operatorname{dim} \xi<n$, 
$$
\mu\left(\xi \cap S^{n-1}\right) \leq \frac{\operatorname{dim} \xi}{n} \mu\left(S^{n-1}\right),
$$
and if the equality holds for some $\xi$, then there exists a subspace $\xi^{\prime}$, which is complementary to $\xi$ in $\mathbb{R}^n$, so that 
$$
\mu\left(\xi^{\prime} \cap S^{n-1}\right)=\frac{\operatorname{dim} \xi^{\prime}}{n} \mu\left(S^{n-1}\right) ,
$$
also holds. 

He–Leng–Li\cite{HLL} first proved origin-symmetric convex polytopes must satisfy the subspace concentration condition. An independent proof was later provided by Xiong\cite{Xi}. Subsequently, Böröczky–Lutwak–Yang–Zhang\cite{BLYZ} showed that the subspace concentration condition is necessary for arbitrary origin-symmetric convex bodies.

If we take $\xi$ to be the one-dimensional subspace spanned by $u$, the subspace concentration condition becomes 
\[
n \frac{\mu(u)+\mu(-u)}{V}\leq 1. 
\]
Since every origin-symmetric convex body has its centroid at the origin, thus \eqref{necessary} refines the above inequality with an extra term $(n+1)^{n-1} |\frac{\mu(u)-\mu(-u)}{V}|^n$. Based on this connection with the subspace concentration condition, it would be interesting to explore refinements for higher-dimensional subspaces. 

In Section 2, we present the proof of the main theorem. In the appendix, We add proofs of two lemmas based on simple but lengthy calculations. 

\section{Proof of Theorem \ref{main}}

In this section we present the proof of Theorem \ref{main}, which consists mainly three steps: the first step is to reduce  \eqref{necessary} to the Schwarz symmetrization of $K$ with respect to $u$. The second step is to  verify  \eqref{necessary} for truncated cones. The final step is to show one can replace the Schwarz symmetrization body of the first step by a truncated cone. 

\subsection{Step 1: Reduction to its Schwarz symmetrization}
\ 

We setup some basic notations first. Denote by $K$ an $n$-dimensional convex body,i.e., a compact convex set of $\mathbb{R}^n$ with nonempty interior.   
Its support function $h_K: S^{n-1}\to \mathbb{R}$ is defined as 
\[
h_K(u)=\sup_{x\in K} x\cdot u. 
\]
Set $F_K(u)=\{ y\in K\quad |\quad  y \cdot u=h_K(u)\}$. Fix a unit vector $u \in \mathbb{S}^{n-1}$. Let 
\[
A_K(t)=\mathcal{H}^{n-1}(K\cap \{x\cdot u=t\})
\] denote the area of a hyperplane section of $K$. 
Clearly, one can think $A_K(t)$ is defined on $t \in\left[-h_K(-u), h_K(u)\right]$. 

For a give convex body $K$, we form a new set $K'$ by replacing each $n-1$-dimensional slice $K\cap \{x\cdot u=t\}$ by an $(n-1)$-dimensional disks of same area centered along $u$-axis. 

More precisely, $\forall t \in\left[-h_K(-u), h_K(u)\right]$, 
\begin{align}
 K^{\prime} \cap\left\{ x \cdot u=t\right\} =\{x \in \mathbf{R}^n:\|x-t u\| \leq\left(\frac{A_K(t)}{\omega_{n-1}}\right)^{\frac{1}{n-1}},\quad  x \cdot u=t \}. 
\end{align}
Here $\omega_{n-1}$ is the volume of unit $(n-1)$-dimensional ball.

Such a procedure is called the \textbf{Schwarz symmetrization} of $K$ with respect to $u$. 
It is a well-known fact that such $K'$ obtained is also a convex body. 
Moreover since $A_{K}(t)=A_{K'}(t)$ by the construction, we have 
\[
V\left(K^{\prime}\right)=\int_{-h_K(-u)}^{h_K(u)} A_{K^{\prime}}(t) \mathrm{d} t=\int_{-h_K(-u)}^{h_K(u)} A_K(t) \mathrm{d} t=V(K).
\]

The reduction of $K$ to its Schwarz symmetrization is due to the following proposition. 

\begin{proposition} \label{P1}
    Let $K$ be a convex body with centroid at the origin, let $K^{\prime}$ be its Schwarz symmetrization with respect a unit vector $u\in\mathbb{S}^{n-1}$. Denote the corresponding cone volume measures by $\mu_K$ and $\mu_{K'}$ respectively, then 
$\mu_K( u)=\mu_{K'}( u)$, $\mu_K( -u)=\mu_{K'}( -u)$. 
\end{proposition}

\begin{proof}
It follows from the definition of the Schwarz symmetrization that 
\[
h_K(\pm u)=h_{K'}(\pm u) \quad \text{and} \quad A_{K}( h_K(\pm u))=A_{K'} (h_{K'}(\pm u)).
\]
We shall show that the centroid of $K'$ is also at the origin, from which the conclusion follows. 

In the $u$ direction, we have 
\begin{align} \notag
c\left(K^{\prime}\right) \cdot u=&\frac{1}{V\left(K^{\prime}\right)} \int_{K^{\prime}} x \cdot u d x=\frac{\int_{-h_K(-u)}^{h_K(u)} t A_{K^{\prime}}(t) \mathrm{d} t}{\int_{-h_K(-u)}^{h_K(u)} A_{K^{\prime}}(t) \mathrm{d} t} \\ \notag
&=\frac{\int_{-h_K(-u)}^{h_K(u)} t A_K(t) \mathrm{d} t}{\int_{-h_K(-u)}^{h_K(u)} A_K(t) \mathrm{d} t}=c(K) \cdot u=0. 
\end{align}
The coordinate of the centroid $c(K')$ along other directions perpendicular to $u$ is clearly zero.
\end{proof}

\subsection{Step 2: Verification of \eqref{necessary} for truncated cones}
\ 

In this subsection, we verify \eqref{necessary} for truncated cones. 

By previous section, we may assume $K$ is Schwarz symmetrized with respect to $u$.
Now we further assume that $K=\operatorname{conv}\{F_{K}(u)\cup F_{k}(-u)\}$, i.e., when $K$ happens to be a truncated cone. 

\begin{proposition} \label{P2}
Let $u$ be any fixed unit vector, and assume $K$ is a truncated cone with center axis along the $u$ direction. Denote the cone volume measure of $K$ by $\mu$, set $x=\frac{\mu(u)}{V(K)}$ and $y=\frac{\mu (-u)}{V(K)}$, then 
    \begin{align} \label{algebraic}
    n(x+y)+(n+1)^{n-1}(x-y)^{n}\leq 1.
    \end{align}
    Moreover the equality holds if and only if either $(x,y)=(0, \frac{1}{n+1})$, $(x,y)=(\frac{1}{n+1}, 0)$ or $x=y=\frac{1}{2n}$. 
\end{proposition}

\begin{proof}
First of all, if $K$ is a right cylinder with axis along the $u$ direction, it is easy to find $x=y=\frac{1}{2n}$ and \eqref{algebraic} trivially holds.  

Secondly, a simple computation shows that $x,y$ only depend on the ratio of radii of $F_{K}(u)$ and $F_{K}(-u)$ and does not depend on the height. By reversing the direction of $u$, we may assume that the top base of $K$ is larger than the bottom base of $K$. So we may assume that $K$ lies between $x\cdot u=t, (t>1)$ and $x\cdot u=1$ with lower radius being $1$ and upper radius being $t$. Hence $t$ can be regarded as the ratio of radii of bases of $K$. $t=1$ corresponds to a right cylinder and $t=\infty$ corresponds to a right circular cone. 

Then
\[
V(K)  =\int_1^t \omega_{n-1} h^{n-1} \mathrm{d} h=\frac{\omega_{n-1}}{n}\left(t^n-1\right).
\]

\[
c(K) \cdot u  =\frac{1}{V(K)} \int_K x \cdot u \mathrm{~d} x=\frac{\omega_{n-1}\int_1^t  h^{n} \mathrm{d} h}{\omega_{n-1}\int_1^t h^{n-1} \mathrm{d} h} =\frac{\frac{1}{n+1}\left(t^{n+1}-1\right)}{\frac{1}{n}\left(t^n-1\right)}.
\]

Consequently, 
\begin{align} \label{xy}
x=\frac{t^{2 n}-(n+1) t^n+n t^{n-1}}{(n+1)\left(t^n-1\right)^2} \quad \text{and} \quad  \quad y=\frac{n t^{n+1}-(n+1) t^n+1}{(n+1)\left(t^n-1\right)^2}.
\end{align}

We present two lemmas whose proofs will be postponed in the appendix. 

\begin{lemma} \label{L1}
Let $x,y$ be as in \eqref{xy}, then
    \begin{align} \notag
    x\leq \frac{1}{n+1}, \quad y\leq \frac{1}{n+1}, \quad x+y\leq \frac{1}{n}.
\end{align}
\end{lemma}

\begin{lemma} \label{L2}
        \begin{align} \label{keyinequality}
\frac{\left(t^n-1\right)^{2 n}-n^2 t^{n-1}(t-1)^2\left(t^n-1\right)^{2 n-2}}{\left(t^{2 n}-n t^{n+1}+n t^{n-1}-1\right)^n} \geq 1, \quad \forall t\geq 1.
\end{align}
\end{lemma}

Now let us go back to the inequality \eqref{algebraic}. Using \eqref{xy}, it amounts to prove  
\begin{align} \label{xy'}
    n A+(1-B)^n \leq 1,
\end{align}
where
\[
 A=\frac{n t^{n-1}(t-1)^2}{\left(t^n-1\right)^2}  \quad 
B=\frac{n t^{n+1}-2 t^n-n t^{n-1}+2}{\left(t^n-1\right)^2}.
\]
A simplification implies that \eqref{xy'} is equivalent to \eqref{keyinequality}. This finishes the proof. Based on the proof of Lemma \ref{L2}, it is easy to see that inequality in \eqref{xy'} is strict and it holds only if $t=\infty$ which corresponds to one of $x, y$ vanishes. 

\end{proof}

\subsection{Step 3: Reduction to a truncated cone}
\ 

After the first step, we get a Schwarz symmetrization $K'$. Suppose $K'\neq \operatorname{conv}\{F_{K'}(u)\cup F_{K'}(-u)\}$. Since $K'$ is a convex set, we must have $K'\supset \operatorname{conv}\{F_{K'}(u)\cup F_{K'}(-u)\}$ and
thus $V(K')>V(\operatorname{conv}\{F_{K'}(u)\cup F_{K'}(-u)\})$. 

Note $F_{K'}(u)$ and $F_{K'}(-u)$ are both $(n-1)$-disks (possibly degenerating to a point), which shall be referred as the top and the bottom of $K'$ respectively. There exists a unique $(n-1)$-dimensional disk $D_0$ in the hyperplane $u \cdot x=h_{K^{\prime}}(u)$ that has the same center as $F_{K^{\prime}}(u)$ but a larger radius such that
$$
V\left(\operatorname{conv}\left\{D_0 \cup F_{K^{\prime}}(-u)\right\}\right)=V\left(K^{\prime}\right) .
$$
Similarly, there exists a unique $(n-1)$-dimensional disk $D_1$ in the hyperplane $u \cdot x=-h_{K^{\prime}}(-u)$ that has the same center as $F_{K^{\prime}}(-u)$ but a larger radius such that
$$
V\left(\operatorname{conv}\left\{F_{K^{\prime}}(u)\cup D_1\right\}\right)=V\left(K^{\prime}\right) .
$$
Set $K_0=\operatorname{conv}\{ D_0\cup F_{K'}(-u)\}$ and $K_1=\operatorname{conv}\{ F_{k'}(u)\cup D_1\}$. $K_0$ is a truncated cone with top $D_0$ and bottom $F_{K'}(-u)$ and $K_1$ is a truncated cone with top $F_{K'}(u)$ and bottom $D_1$. Let $K_t, t\in[0,1]$ be a continuous family of truncated cones trapped in between $x\cdot u= h_{K'}( u)$ and $x\cdot u=-h_{K'}(-u)$ with the top being $(1-t)D_0+tF_{K'}(u)$ and $V(K_t)=V(K')$. With this volume constraint, the bottom of this family of truncated cones forms also a continuous deformation of from $D_1$ to $F_{K'}(-u)$. 

\begin{figure}[htbp]
    \centering
    \includegraphics[width=0.75\linewidth]{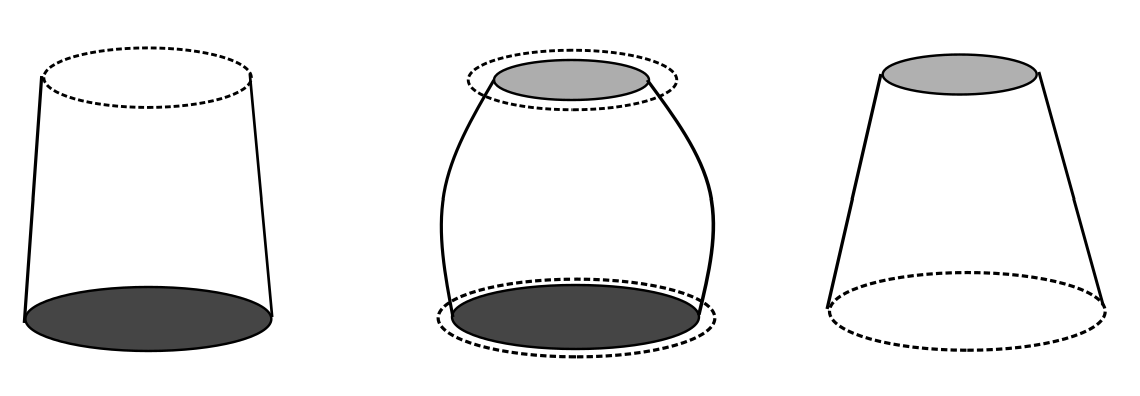}
    \caption{Left: $K_0$, Middle: $K'$, Right: $K_1$}
\end{figure}

We \textbf{claim} that $c(K_0)\cdot u\geq 0\geq c(K_1)\cdot u.$
In fact, there exists a unique $h_0\in [-h_{K'}(-u), h_{K'}(u)]$ such that 
\[
A_{K'}(h)\geq A_0(h), \quad \forall h\leq h_0, 
\]
and 
\[
A_{K'}(h)\leq A_0(h), \quad \forall h\geq h_0.
\]

Thus 
\[
\int_{-h_{K'}(-u)}^{h_0} h (A_{K'}(h)-A_0(h)) dh \leq h_0 \int_{-h_{K'}(-u)}^{h_0} (A_{K'}(h)-A_0(h)) dh,
\]
and 
\[
\int_{h_0}^{h_{K'}(u)} h (A_{K'}(h)-A_0(h)) dh \leq h_0  \int_{h_0}^{h_{K'}(u)}  (A_{K'}(h)-A_0(h)) dh.
\]

Adding both sides of above two inequalities and using 
\[
V(K_0)=\int_{-h_{K'}(-u)}^{h_{K'}(u)} A_0(h) dh =V(K')=\int_{-h_{K'}(-u)}^{h_{K'}(u)} A_{K'}(h) dh,
\]
we find that 
\[
c(K_0)\cdot u=\int_{-h_{K'}(-u)}^{h_{K'}(u)} h A_0(h) dh \geq \int_{-h_{K'}(-u)}^{h_{K'}(u)} h A_{K'}(h) dh=c(K')\cdot u=0.
\]
The second inequality follows verbatim. 

Therefore by continuity, there exists $t'\in(0,1)$ such that $c(K_{t'})\cdot u=0$. 
Comparing $K'$ with $K_{t'}$, we find both the top and the bottom of $K'$ are strictly smaller than the top and the bottom of $K_{t'}$, yet $V(K')=V(K_{t'})$. 
Therefore 
\[
x'=\frac{\mu_{K'}(u)}{V(K')}< \frac{\mu_{t'}(u)}{V(K_{t'})}=x_{t'},
\]
and 
\[
y'=\frac{\mu_{K'}(-u)}{V(K')}<\frac{\mu_{t'}(-u)}{V(K_{t'})}=y_{t'}.
\]

Set $\Psi(x,y)= n(x+y)+(n+1)^{n-1}(x-y)^{n}$, we have
\[
\frac{\partial \Psi}{\partial x}(x,y)=n+n(n+1)^{n-1}(x-y)^{n-1},
\]
and 
\[
\frac{\partial \Psi}{\partial y}(x,y)=n-n(n+1) ^{n-1}(x-y)^{n-1}. 
\]
In view of the range \eqref{range} for $x_{t'}$ and $y_{t'}$, we have $\frac{\partial \Psi}{\partial x}\geq 0$ and $\frac{\partial \Psi}{\partial y}\geq 0$, then
\[
\Psi(x', y')< \Psi(x_{t'}, y_{t'}). 
\]
This implies one can reduce the verification of \eqref{necessary} for $K'$ to the verification of \eqref{necessary} for the truncated cone $K_{t'}$, which is done in Proposition \ref{P2}.

\subsection{Completion of the proof}
\ 

Now we are in a position to combine the results achieved so far. Given an arbitrary convex body $K$ with centroid at the origin, Step 1 shows that replacing $K$ by its Schwarz symmetrization $K^{\prime}$ with respect to the direction $u$ does not change the values of $x=\frac{\mu(u)}{V}$ and $y=\frac{\mu(-u)}{V}$. Then, by Step 3, we may further replace $K^{\prime}$ by a truncated cone of the same volume, which yields a larger value of
$$
\Psi(x, y)=n(x+y)+(n+1)^{n-1}(x-y)^n.
$$
However, $\Psi(x,y)\leq 1$ by Step 2. By considering a pair of opposite directions, we get the desired \eqref{necessary}.

The equality case occurs if and only if $K'$ is a right cylinder or a right circular cone. So $A_{K}^{\frac{1}{n-1}}(t)$ is a linear function, by the virtue of the Brunn-Minkowski inequality, the original $K$ must be either an oblique prism or a cone.

\section{Appendix}
In this appendix, we provide proofs of Lemma \ref{L1} and Lemma \ref{L2}.
Both proofs follow a similar approach: we successively take derivatives and factor out common powers of t to reduce to a polynomial, which is simple enough to verify that it is nonnegative. For reader's convenience, we rewrite two lemmas. 

\begin{lemma} 
Let $x,y$ be given as 
\begin{align}\notag 
x=\frac{t^{2 n}-(n+1) t^n+n t^{n-1}}{(n+1)\left(t^n-1\right)^2} \quad \text{and} \quad  \quad y=\frac{n t^{n+1}-(n+1) t^n+1}{(n+1)\left(t^n-1\right)^2},
\end{align}
then $\forall t\geq 1$, we have 
\begin{align} \label{range}
    x\leq \frac{1}{n+1}, \quad y\leq \frac{1}{n+1}, \quad x+y\leq \frac{1}{n}.
\end{align}
\end{lemma}

\begin{proof}[Proof of Lemma \ref{L1}]
\ 

Note $x \leq \frac{1}{n+1} \Leftrightarrow 1+n t^{n+1}-(n+1) t^n \leq\left(t^n-1\right)^2$. The latter is equivalent to 
$$
f_1(t):=t^n-n t+(n-1) \geq 0, \forall t \geq 1. 
$$
This follows from the fact that $f_1^{\prime}(t)=n t^{n-1}-n \geq 0, \forall t \geq 1$ and $f_1(1)=0$.

Similarly,
$y \leq  \frac{1}{n+1} \Leftrightarrow t^{2 n}+n t^{n-1}-(n+1) t^n \leq\left(t^n-1\right)^2 $ and the latter is equivalent to 
$$
f_2(t)=(n-1) t^n-n t^{n-1}+1 \geq 0, \forall t \geq 1.
$$
This follows from the fact that $f_2^{\prime}(t)=n(n-1) t^{n-1}-n(n-1) t^{n-2} \geq 0, \forall t \geq 1$ and  $f_2(1)=0$.

For $ (x+y) \leq \frac{1}{n}$, as noted in the introduction, this follows from the subspace concentration condition for the one-dimensional subspace spanned by $u$. Here we also offer a direct proof. It is equivalent to show 
\[
 \frac{n}{n+1} \cdot \frac{\left(t^n-1\right)^2+n t^{n+1}-2 n t^n+n t^{n-1}}{\left(t^n-1\right)^2} \leq 1.
 \]
After simplification, it is left to show
\[
g(t)=  t^{2 n}-n^2 t^{n+1}+\left(2 n^2-2\right) t^n-n^2 t^{n-1}+1 \geq 0, \quad \forall t \geq 1.
\]
Note $g(1)=0, g^{\prime}(1)=0$, we shall show that 
\[
g^{\prime}(t)=2 n t^{2 n-1}-n^2(n+1) t^n+n\left(2 n^2-2\right) t^{n-1}
-n^2(n-1) t^{n-2} \geq 0, \forall t \geq 1.
\]
Dividing by $nt^{n-2}$, one just needs to show 
\[
h(t)=2t^{n+1}-n(n+1) t^2+(2n^2-2) t-n(n-1)\geq0 \quad \forall t\geq 1. 
\]
Note again $h(1)=0$ and $h'(t)=2(n+1)[t^n-nt+n-1]$. So $h'(1)=0$.  Take derivative again, we find that $h''(t)=2(n+1)n (t^{n-1}-1)\geq 0$, $\forall t\geq 1$. The desired inequality follows. 
\end{proof}

\begin{lemma} 
Assume $n\geq 3$, then for $t>1$, we have 
\begin{align} \notag 
\frac{\left(t^n-1\right)^{2 n}-n^2 t^{n-1}(t-1)^2\left(t^n-1\right)^{2 n-2}}{\left(t^{2 n}-n t^{n+1}+n t^{n-1}-1\right)^n} \geq 1.
\end{align}
\end{lemma}

\begin{proof}[Proof of Lemma \ref{L2}]
\ 

Set $F=\left(t^n-1\right)^{2 n}-n^2 t^{n-1}(t-1)^2\left(t^n-1\right)^{2 n-2}$ and $
G=\left(t^{2 n}-n t^{n+1}+n t^{n-1}-1\right)^n$.

We \textbf{claim} that $\frac{F}{G}$ is non increasing. Given that, notice that 
\[
\lim_{t\to \infty} \frac{F(t)}{G(t)}=1,
\]
and thus the inequality follows. 

To show that $\frac{F}{G}$ is non increasing, it suffices to show 
\[
\frac{F'}{F}\leq \frac{G'}{G}.
\]
We have
\[
\frac{F^{\prime}}{F}=n^2 \frac{t^{n-2}\left[2 t\left(t^n-1\right)^2-2 n(n-1) t^n(t-1)^2-(t-1)\left(t^n-1\right)((n+1) t-(n-1))\right]}{\left(t^n-1\right)\left[\left(t^n-1\right)^2-n^2 t^{n-1}(t-1)^2\right]},
\]
and
\[
\frac{G^{\prime}}{G}=n^2 \frac{t^{n-2}\left(2 t^{ n+1}-(n+1) t^2+(n-1) \right)}{t^{2 n}-n t^{n+1}+n t^{n-1}-1}.
\]

After simplification, it amounts to show $\forall t \geq 1$, we have 
\begin{align} \notag
(2 n-2)  &t^{3 n}+\left(2 n^3-2 n^2-2 n+2\right) t^{2 n+1}+\left(n^3-n\right) t^{2 n-2}+\left(2 n^3-2 n^2-4\right) t^{n+1} \\ \notag
& +\left(2 n^2-4 n-6\right) t^n+\left(n^3-2 n^2+n\right) t^{n-2}+(2 n+2) t \\\notag
& -\{ (2 n-2)+\left(2 n^3-2 n^2-2 n+2\right) t^{n-1}+\left(n^3-n\right) t^{n+2} \\\notag
& +\left(2 n^3-2 n^2-4\right) t^{2 n-1}+\left(2 n^2-4 n-6\right) t^{2 n}+\left(n^3-2 n^2+n\right) t^{2 n+2} \\ \notag
& +(2 n+2) t^{3 n-1} \}\geq 0.
\end{align}

Denote the polynomial on the left hand side by $p_1(t)$.  By direction computation, we find 
\[
p_1^{(m)}(1)=0, m=0,1,2.
\]
Let $p_2(t)=\frac{p_1^{(2)}(t)}{t^{n-4}}$, we again find that 
\[
p_2^{(m)}(1)=0, \quad m=0, 1, 2, 3, 4. 
\]
Let $p_3(t)=\frac{p_2^{(5)}(t)}{t^{n-5}}$, then 
\begin{align} \notag 
p_3(1)&=140n^2 (n-1)^2 (n+1)^2 (n-2)>0, \\ \notag
p_3'(1)&=140n^2 (n-1)^2 (n+1)^2 (n-2)(5n+1)>0, \\ \notag
p_3''(1)&=n^2 (n-1)^2 (n+1)^2 (n-2)(1568n^2+1344n+336)>0, \\ \notag
p_3^{(3)}(1)&=n^2 (n-1)^2 (n+1)^2 (n-2)(n+1/2)(2352n^2+1456n+1344)>0, \\ \notag
p_3^{(4)}(1)&=n^2 (n-1)^2 (n+1)^2 (n-2)(n+1/2)(2976n^3-848n^2+2608n+1056)>0, \\ \notag
p_3^{(5)}(1)&=n^2 (n-1)^2 (n+1)^2 (n+1/2)(n-1/2)(n-1/3) \\ \notag
& \cdot (n-2)(3552n^2-5742n-1728)>0.
\end{align}
Set $p_4(t)=\frac{p_3^{(5)}(t)}{t^{n-4}}$, then 
\begin{align} \notag 
p_4'(t)&\equiv (2n-2)3n(3n-1)(2n+2)(2n+1)2n\\ \notag
 &\cdot (2n-1)(2n-2)(n+2)(n+1)n(n-1)(n-2)>0.
\end{align}
Based on these computations, $p_1(t)\geq 0$, $\forall t\geq 1$. In fact, it is easy to see the inequality is strict. 

Note if $n=2$, we have the identity
\[
\frac{\left(t^n-1\right)^{2 n}-n^2 t^{n-1}(t-1)^2\left(t^n-1\right)^{2 n-2}}{\left(t^{2 n}-n t^{n+1}+n t^{n-1}-1\right)^n} \equiv  1,
\]
which explains why, in dimension two, if equality holds the convex body $K$ can generally be a trapezoid.
\end{proof}

\bibliographystyle{amsplain}
\bibliography{references}
\end{document}